\documentclass [12pt]{article}

\usepackage{amsmath,amssymb}
\numberwithin{equation}{section}
\usepackage{amsthm}
\newtheorem{theorem}{Theorem}
\newtheorem{proposition}[theorem]{Proposition}
\newtheorem{remark}[theorem]{Remark}

\newtheorem{definition}[theorem]{Definition}
\newtheorem{corollary}[theorem]{Corollary}

\newcommand{\RR}{\ensuremath{\mathbb{R}}}
\newcommand{\Doo}{\ensuremath{\Delta_{00}}}
\newcommand{\D}{\ensuremath{\Delta_{--}}}
\newcommand{\bd}[1] {\ensuremath{\partial{#1}}}

\newcommand{\Oplus} {\ensuremath{\mathop{\oplus}\limits}}
\newcommand{\Sum} {\ensuremath{\mathop{\sum}\limits}}

\newcommand{\ddnu}[1] {\ensuremath{\frac{\partial
      #1}{\partial \nu}}}
\newcommand{\Int}{\ensuremath{\mathop{\int}\limits}}
\newcommand{\na}{\noalign}
\newcommand{\del}{\nabla}
\newcommand{\vb}{\bar{v}}
\newcommand{\gb}{\bar{g}}
\newcommand{\mb}{\bar{m}}

\newcommand{\coker}{\mathrm{coker\ }}
\renewcommand{\bar}[1] {\ensuremath{\overline{#1}}}
\newcommand{\half}{\ensuremath{\frac{1}{2}}}
\newcommand{\ltd}{\ensuremath{L^{2}(D)}}
\newcommand{\intd}{\ensuremath{\Int_{D}}}

\newcommand{\espace}{\vspace*{2mm}}

\begin{document}

\begin{titlepage}
  \title{Discreteness of Transmission Eigenvalues via
    Upper Triangular Compact Operators\footnote{This
      research was supported in part by NSF grant
      DMS-1007447}}

  \author{John Sylvester\\
    Department of Mathematics\\
    University of Washington\\
    Seattle, Washington 98195\\
    U.S.A.  }
  \maketitle
\end{titlepage}
\begin{abstract}
 
  Transmission eigenvalues are points in the spectrum of
  the interior transmission operator, a coupled 2x2 system
  of elliptic partial differential equations, where
  one unknown function must satisfy two boundary
  conditions and the other must satisfy none.\\

  We show that the interior transmission eigenvalues are
  discrete and depend continuously on the contrast by
  proving that the interior transmission operator has
  \textit{upper triangular compact resolvent}, and that
  the spectrum of these operators share many of the
  properties of operators with compact resolvent. In
  particular, the spectrum is discrete and
  the generalized eigenspaces are finite dimensional.\\

  Our main hypothesis is a coercivity condition on the
  \textit{contrast} that must hold only in a neighborhood
  of the boundary.

\end{abstract}

\section{Introduction}

The time-harmonic scattering of waves by a penetrable
scatterer in a vacuum can be modeled with the Helmholtz
equation. The \textit{total wave} \(u\) satisfies the
perturbed Helmholtz equation
\begin{equation}\label{eq:helmm}
 \left(\Delta+k^{2}(1+m)\right) u = 0 \quad \mbox{in} \; \RR^{n}
\end{equation}
where the contrast, \(m(x)\), denotes the deviation of the
square of the index of refraction from the constant
background; i.e. \(n^{2}(x)= 1 + m(x)\) .  The relative
(far field) scattering operator, \(s^{+}\), records the
correspondence between the asymptotics of solutions to of
the free Helmholtz equation to those of
(\ref{eq:helmm}). If the operator \(s^{+}\) has a
nontrivial kernel (null space) at wavenumber \(k\), then we
say that \(k\) is a transmission eigenvalue
\cite{Colton-Kress2} \cite{MR2262743}.  Certain
inverse scattering methods are known to succeed only at
wavenumbers that are not transmission eigenvalues
\cite{Colton-Kress2}. If a scatterer \(m(x)\) is supported
in a bounded domain \(D\), then, if \(k^{2}\) is a transmission
eigenvalue, \(k^{2}\) must be an \textit{interior
  transmission eigenvalue} as defined below. It is
possible to make an extended definition of the scattering
operator \(s^{+}\) such that the two are equivalent
\cite{MR2262743}. A precise relationship between interior
transmission eigenvalues and the scattering operator has
yet to be discovered.

\begin{definition}
  A wavenumber \(k^{2}\) is called an \textit{interior transmission
    eigenvalue} of \(m\) in the domain \(D\) if there exists a non-trivial pair
  \((V,W)\) solving
\begin{eqnarray}
    \nonumber
    \Delta W (x)+k^{2}(1+m)W = 0  \quad  & \mbox{in} & \quad D
    \\\nonumber
    \Delta V + k^{2}V = 0 \quad & \mbox{in} & \quad D
    \\\nonumber
    W=V, \; \frac{\partial W}{\partial \nu}  = \frac{\partial
      V}{\partial \nu} \quad & \mbox{on} & \quad \partial D
  \end{eqnarray}
\end{definition}

If we set \(u = W-V\), \(v = k^{2}V\), and
\(\lambda=-k^{2}\), then the interior transmission
eigenvalue problem can be rewritten as the coupled 2x2
system of elliptic partial differential equations
below. Its distinguishing feature is that \(u\) must
satisfy two boundary conditions and \(v\) satisfies none.
 \begin{eqnarray}\nonumber
   &\left(\Delta-\lambda(1+m)\right)u + mv = 0&
\\\label{eq:24}
 &\left(\Delta-\lambda\right)v = 0&
\\ 
  &u=0\ ; \quad\ddnu{u}=0\quad  \mathrm{on}  \quad \partial D&
 \end{eqnarray}

 Thus the interior transmission eigenvalues are the spectrum of the
 generalized eigenvalue problem:
  \begin{equation}
    \label{eq:1}
    B - \lambda I_{m} := \begin{pmatrix}
                  \Doo&m\\
                     0&\D
                  \end{pmatrix} - \lambda\begin{pmatrix}
                                           (1+m)&0\\
                                               0&1
                                         \end{pmatrix}
  \end{equation}
  Where we use the notation \(\Doo\) and \(\D\) to denote
  the fact that \(u\) must satisfy two boundary conditions
  and \(v\) needn't satisfy any. In other words, we treat
  \(B\) as an unbounded operator on \(\ltd\oplus\ltd\)
  with domain:
\begin{eqnarray}
  \nonumber
  B:H^{2}_{0}(D)\oplus
\left\{v\in\ltd : \Delta v\in\ltd\right\}
\longrightarrow \ltd\oplus\ltd
\end{eqnarray}

The Hilbert space \(H^{2}_{0}\) is the completion
of \(C^{\infty}_{0}\) in the norm
\begin{eqnarray*}
  ||u||_{2}^{2}= ||u||^{2} + ||\Delta u||^{2}
\end{eqnarray*}
and \(||u||\)means the \(\ltd\) norm.  We will always
assume that \hbox{\(\Re(1+m)>\delta>0\)}, so we may also describe the
spectrum of the generalized eigenvalue problem as the
spectrum of \(I_{m}^{-1}B\). Neither \(B\) nor
\(I_{m}^{-1}B\) is self-adjoint; neither has a compact
resolvent, but we will show that, as long as \(D\) is
bounded, both resolvents are \textit{upper triangular
  compact}, and that this is enough to reach most of the
conclusions we know for operators with compact resolvents,
including the facts that the nonzero spectrum is discrete
and consists of eigenvalues of finite multiplicity. Our
main theorem is:

\begin{theorem}\label{sec:th1}
  Suppose that there are real numbers \(m^{*}\ge m_{*}>0\)
  and  a unit complex number \(e^{i\theta}\) in the open
  right half plane (i.e \(\frac{-\pi}{2}<\theta<\frac{\pi}{2}\))
  such that
  \begin{enumerate}
  \item \(\Re(e^{i\theta}m(x))>m_{*}\) in some
    neighborhood \(N\) of \(\bd{D}\), or that \(m(x)\) is
    real in all of \(D\), and satisfies
    \(m(x)\le - m_{*}\) in some
    neighborhood \(N\) of \(\bd{D}\).
  \item \(|m(x)|< m^{*}\) in all of \(D\).
  \item \(\Re(1+m(x))\ge \delta>0\) in all of \(D\).
  \end{enumerate}

  Then the spectrum of \eqref{eq:1} consists of a
  (possibly empty) discrete set of eigenvalues with finite
  dimensional generalized eigenspaces. Eigenspaces
  corresponding to different eigenvalues are linearly
  independent. The eigenvalues and the generalized
  eigenspaces depend continuously on \(m\) in the
  \(L^{\infty}(D)\) topology.
\end{theorem}

We will obtain this as a corollary of 
\begin{theorem}\label{sec:th2}
   Suppose that there are real numbers \(m^{*}\ge m_{*}>0\)
  and  a unit complex number \(e^{i\theta}\)
  such that
  \begin{enumerate}
  \item     \(\Re(e^{i\theta}m(x))>m_{*}\)
    in  some neighborhood of \(\bd{D}\).
  \item \(|m(x)|< m^{*}\) in all of \(D\).
  \end{enumerate}
  Then the spectrum of \(B\) consists of a (possibly
  empty) discrete set of eigenvalues with finite
  dimensional generalized eigenspaces. Eigenspaces
  corresponding to different eigenvalues are linearly
  independent.The eigenvalues and the generalized
  eigenspaces depend continuously on \(m\) in the
  \(L^{\infty}(D)\) topology.
\end{theorem}
\begin{remark}
  If we assume that \(m\) is continuous in the closure of
  \(D\), then we need only state the conditions in item 1
  of each theorem on \(\bd{D}\), and they will necessary
  hold in some neighborhood.
\end{remark}

Our statement about the independence of eigenfunctions
corresponding to different eigenvalues is, of course,
trivially true for any linear operator. We state it
explicitly because it is not obvious in other
formulations of the interior transmission probelm.The
continuity of the eigenvalues and eigenfunctions will
follow from the continuity of the spectral projections in
the proposition below.  We use the subscript in \(B_{m}\)
to denote the dependence on the contrast \(m(x)\).
\begin{proposition}\label{sec:spectralP}
  Let \(A\) represent the operator \(B_{m}\), or
  \(I_{m}^{-1}B_{m}\) and let \(\gamma\) be a bounded
  rectifiable curve in the complex plane that avoids the
  spectrum of \(A\). Then the spectral projection
  \begin{eqnarray}\label{eq:36}
    &P_{\gamma}(m) :=\frac{1}{2\pi i}
    \Int_{\gamma}\left(A-\lambda I\right)^{-1}d\lambda&
\\\na{viewed as an operator valued  function of \(m\in L^{\infty}(D)\)
  with values in the space of operators mapping:\espace}\nonumber
 &P_{\gamma} : \ltd\oplus\ltd \longrightarrow H^{2}_{0}(D)\oplus\ltd&
\end{eqnarray}
is continuous at \(m\).
\end{proposition}

There are many different ways to state the continuity of
eigenvalues and eigenspaces, but all can be infered from
proposition \ref{sec:spectralP} by choosing the curve
\(\gamma\) appropriately. Typically, the most useful
choice is a small circle surrounding an isolated
eigenvalue. The continuity of \(P_{\gamma}\) implies, for
example, that the dimension of its range is constant for
small perturbations, so that a simple eigenvalue must
remain simple, and the total algebraic multiplicity of the
eigenvalues inside \(\gamma\) cannot change.

\begin{remark}
  If we modify the background of our \textit{interior
  transmission eigenvalue} to be a variable real valued
index of refraction, \(n(x)\), that is bounded away from
zero, and even replace the Laplacian with a real
divergence form operator, i.e.

\begin{definition}
  A wavenumber \(k^{2}\) is called an \textit{interior transmission
    eigenvalue} of \(m\) in the domain \(D\) with real background \((n,\gamma)\)
  if there exists a non-trivial pair
  \((V,W)\) solving
\begin{eqnarray}
    \nonumber
    \nabla\cdot\gamma\nabla W (x)+k^{2}n^{2}(x)(1+m(x))W = 0  \quad  & \mbox{in} & \quad D
    \\\nonumber
    \nabla\cdot\gamma\nabla V + k^{2}n^{2}(x)V = 0 \quad & \mbox{in} & \quad D
    \\\nonumber
    W=V, \; \frac{\partial W}{\partial \nu}  = \frac{\partial
      V}{\partial \nu} \quad & \mbox{on} & \quad \partial D
  \end{eqnarray}
\end{definition}

As long as \(n(x)\ge \delta>0\) and \(\gamma\) is real,
positive definite, and satisfyies \(\gamma\ge \delta
I\) both theorem \ref{sec:th1} and theorem \ref{sec:th2}
continue to hold.The proofs of all the estimtes are
correct line by line with just the obvious
substitutions. The resolvent in (\ref{eq:1}) becomes


\begin{equation}
    \nonumber
    B - \lambda I_{m} := \frac{1}{n^{2}(x)}\begin{pmatrix}
                  \nabla\gamma\nabla_{00}&n^{2}m\\
                     0&\nabla\gamma\nabla_{--}
                  \end{pmatrix} - \lambda\begin{pmatrix}
                                           (1+m)&0\\
                                               0&1
                                         \end{pmatrix}
  \end{equation}

  which enjoys the same duality (e.g. (\ref{eq:10})). As the
  compactness properties follow directly from the
  estimates, they are also the same.

\end{remark}

The transmission eigenvalue was introduced by Colton and
Monk in 1988 \cite{MR934695}. In 1989, Colton,
Kirsch, and P{\"a}iv{\"a}rinta \cite{MR1019312}
proved that the set of real transmission eigenvalues was
discrete. That this set was non-empty was first proved in
2008 by P{\"a}iv{\"a}rinta and the author
\cite{MR2438784}, under the hypothesis that \(m\) was
large enough. In 2010, Cakoni, Gintides, and Haddar
removed that restriction and showed that the set of real
transmission eigenvalues was infinite \cite{MR2596553}.\\

Our results differ from previous work because we assume
relatively little about the contrast \(m\) in the interior
of \(D\). Most of the previous work that we are aware of
requires that \(m(x)>m_{*}>0\) or \(m(x)<-m_{*}<0\) in the
whole domain. An exception is \cite{MR2596549}, which only
assumes \(m(x)\ge 0\) and allows cavities (i.e. \(m(x)=0\)
on an open subset of \(D\)).Some of these results on real
transmission eigenvalues have been extended to general
elliptic operators in
\cite{MR2745799}.\\

\section{A Priori Estimates}

Because \(v\) in \eqref{eq:24} satisfies no boundary
conditions, we don't have a direct estimate in terms of
\(||g||\). However, we still have the local elliptic
estimates. We prove a simple version of these estimates
below to show that, for large positive \(\lambda\), Most
of \(||v||\) is concentrated near the boundary. The function
\(\rho\) in the proposition will be either \(1\) or
\((1+m)\) in our applications.

\begin{proposition}\label{sec:interior-est}
 Suppose that \(\Re(\rho)>\delta>0\) and \(\phi(x)\in
 C_{0}^{\infty}(D)\) be real valued, with \(0\le\phi\le 1\) 
 (in our applications, we will take \(\phi(x)=1\) outside
 a neighborhood \(N\) of \(\bd{D}\)). If
 \begin{eqnarray*}
   \left(\Delta-\lambda\rho\right)v = g \quad \mathrm{ in\ } D
 \end{eqnarray*}
  then, there is a constant \(K(\phi,\delta)\) such
  that, for sufficiently large positive \(\lambda\),
  
 \begin{eqnarray}
   \label{eq:17}
   ||\phi v||^{2}&\le&
   \frac{K}{\lambda-K}\left(||(1-\phi)v||^{2}
     +||\phi g||^{2}\right) 
\\\label{eq:31}
 ||v||^{2} &\le& K\left(||(1-\phi)v||^{2}
+ \frac{||\phi g||^{2}}{\lambda-K}\right)
    \\\label{eq:32}
  ||\del(\phi v)||^{2}&\le& K\left(||v||^{2}+
 +||\phi g||^{2}\right)
 \end{eqnarray}
\end{proposition}

\begin{proof}
 
\begin{eqnarray}\nonumber
  \intd \vb \phi^{2} \left(\Delta-\lambda\rho\right)v 
  &=& \intd\phi^{2}\vb g
\\\nonumber
  -\int \del (\phi^{2}\vb)\cdot\del v
  -\lambda\int\rho|\phi v|^{2} 
  &=& \intd\phi\vb\phi g
\\\nonumber
 -\int \del (\phi\vb)\cdot\phi\del v -
 \phi\vb\del\phi\cdot\del v -\lambda\int\rho|\phi v|^{2} 
  &=& 
\\\nonumber
-\int \del (\phi\vb)\cdot\del(\phi v) + 
\int  \del (\phi\vb)\cdot v\del\phi -
 \phi\vb\del\phi\cdot\del v -\lambda\int\rho|\phi v|^{2}
 &=& 
\\\nonumber
 -\int |\del(\phi v)|^{2} + 
 \int |\del\phi|^{2}\;|v|^{2} +
\half \int \del\phi^{2}\cdot ( v\del\vb\cdot - \vb\del v)
-\lambda\int\rho|\phi v|^{2}
  &=&
\\\na{Taking real parts removes the third term on the left hand side}\nonumber
 -\int |\del(\phi v)|^{2} + 
 \int |\del\phi|^{2}\;|v|^{2} 
-\lambda\int\rho|\phi v|^{2}
  &=&  \Re\left(\intd\phi\vb\phi g\right)
\end{eqnarray}

Rearranging yields
  \begin{eqnarray}
\nonumber
  ||\del(\phi v)||^{2}
+ \lambda\int\rho|\phi v|^{2}
 &=&
\int |v|^{2}|\del\phi|^{2}
-  \Re\left(\intd\phi\vb\phi g\right)
\\ \nonumber
 ||\del(\phi v)||^{2}
+ \lambda\delta||\phi v||^{2}
 &\le&
 K(\phi)\left(||v||^{2}+
  ||\phi g||^{2}\right) 
\\\na{which immediately yields (\ref{eq:32}) and also}
\nonumber
  ||\phi v||^{2} &\le& \frac{K}{\lambda\delta}\left(||v||^{2}+
  ||\phi g||^{2}\right)
\\\nonumber
&\le&\frac{2K}{\lambda\delta}
\left(||\phi v||^{2}+ ||(1-\phi) v||^{2} +
  ||\phi g||^{2}\right)
\end{eqnarray}
which becomes \eqref{eq:17} after subtracting the term
\(||\phi v||^{2}\) from both sides.

Adding \(||(1-\phi)v||^{2}\) to both sides yields
\eqref{eq:31}.
\end{proof}

\begin{corollary}
  Suppose that \(\phi\), \(v\), and \(\rho\)
    satisfy the hypothesis of proposition
    \ref{sec:interior-est}, that, for some unit complex
    number \(e^{i\theta}\) and  some
    neighborhood \(N\) of \(\bd{D}\),
    \(\Re(e^{i\theta}m(x))>m_{*}\), and that \(\phi(x)=1\) in
    \(D\setminus N\). Then, for sufficiently large
    positive \(\lambda\),
    \begin{eqnarray}
      \label{eq:18}
      ||v||^{2} &\le& K \left( \left|\int m |v|^{2}\right|
        + \frac{||\phi g||^{2}}{\lambda-K}\right)
\\\na{and}\label{eq:19}
      \left|\int m\phi^{2} |v|^{2}\right| &\le&
      \frac{K}{\lambda-K}
      \left(\left|\int(1-\phi^{2})m|v|^{2}\right|
        +||\phi g||^{2}\right) 
    \end{eqnarray}
\end{corollary}
\begin{proof}
  \begin{eqnarray}
  \nonumber
\int|v|^{2} &\le& \left|\int(1-\phi^{2})|v|^{2}\right|
+  \left|\int\phi^{2}|v|^{2}\right|
\\\nonumber
 &\le&
\left|\int\Re\left(\frac{e^{i\theta}m}{m_{*}}\right)(1-\phi^{2})|v|^{2}\right|
+  \left|\int\phi^{2}|v|^{2}\right|
\\\nonumber
 &\le&
\left|\Re\left(\frac{e^{i\theta}}{m_{*}}\int m(1-\phi^{2})|v|^{2}\right)\right|
+  \left|\int\phi^{2}|v|^{2}\right|
\\\nonumber
&\le&
\left|\frac{1}{m_{*}}\int m(1-\phi^{2})|v|^{2}\right|
+  \left|\int\phi^{2}|v|^{2}\right|
\\\nonumber
&\le&
\left|\frac{1}{m_{*}}\int m|v|^{2}\right| +
\left|\frac{1}{m_{*}}\int\phi^{2} m|v|^{2}\right| 
+  \left|\int\phi^{2}|v|^{2}\right|
\\\nonumber
&\le&
\left|\frac{1}{m_{*}}\int m|v|^{2}\right| +
\frac{m_{*}+m^{*}}{m_{*}}||\phi v||^{2}
\\\na{Applying (\ref{eq:17}),}\nonumber
&\le&
\left|\frac{1}{m_{*}}\int m|v|^{2}\right| +
\left(\frac{m_{*}+m^{*}}{m_{*}}\right)
\frac{K}{\lambda-K}\left(||(1-\phi) v||^{2}
     +||\phi g||^{2}\right)
\\\nonumber
&\le&\left|\frac{1}{m_{*}}\int m|v|^{2}\right| +
\left(\frac{m_{*}+m^{*}}{m_{*}}\right)
\frac{K}{\lambda-K}\left(||v||^{2}
     +||\phi g||^{2}\right)
\\\na{\espace so, with a different constant and \(\lambda\) large
  enough, we obtain \eqref{eq:18}. \espace}\nonumber
||v||^{2}&\le&\widetilde{K}\left(\left|\int
 m|v|^{2}\right| + \frac{||\phi g||^{2}}{\lambda-K}\right)
\end{eqnarray}
We make a similar calculation to establish \eqref{eq:19}
\begin{eqnarray*}
   \left|\int\phi^{2} m|v|^{2}\right| &\le&
   m^{*}\int\phi^{2}|v|^{2}
\\\na{Applying (\ref{eq:17}) gives\espace}
&\le&\frac{K}{\lambda-K}\left(\int(1-\phi)^{2}|v|^{2}
     +||\phi g||^{2}\right)
\\\na{Because \(0\le\phi\le 1\)\espace}
&\le&\frac{K}{\lambda-K}\left(\int(1-\phi^{2})|v|^{2}
     +||\phi g||^{2}\right)
\\
&\le&\frac{K}{\lambda-K}\left(
\left|\int\Re\left(\frac{e^{i\theta}m}{m_{*}}\right)(1-\phi^{2})|v|^{2}\right|
     +||\phi g||^{2}\right)
\\
&\le&\frac{K}{\lambda-K}\left(\frac{1}{m_{*}}\left|\int
      m(1-\phi^{2})|v|^{2}\right|
     +||\phi g||^{2}\right)
\end{eqnarray*}
\end{proof}

Next, we derive some a priori estimates for the resolvent
of \(B\). Suppose that \(\lambda\) is large enough, that
\begin{eqnarray}
    \label{eq:2}
    (\Delta-\lambda) u + mv &=& f
\\\label{eq:4}
    (\Delta-\lambda) v      &=& g
  \end{eqnarray}
and that \(u\) and \(\ddnu{u}\) vanish on
  \(\bd{D}\). Multiplying the complex conjugate of
  \eqref{eq:4} by \(u\) yields
\begin{eqnarray}\nonumber
  \int u(\Delta-\lambda)\vb &=& \int\gb u
  \\\na{and integrating by parts}\label{eq:5}
  \int \vb(\Delta-\lambda)= &=& \int\gb u
  \\\na{Multiplying \eqref{eq:2} by \(\bar{v}\) and
    inserting \eqref{eq:5} yields}\label{eq:12}
  \int m|v|^{2} &=& \int f\vb - \int \gb u
  \\\na{from which we conclude (using \eqref{eq:18}) that}
    \label{eq:7}
  ||v||^{2}&\le& K\left(||f||^{2} +
    \frac{||g||^{2}}{\lambda-K}  +
||g||\;||u||\right)
\end{eqnarray}
\newcommand{\ub}{\bar{u}}
Next we multiply \eqref{eq:2} by \(\ub\)

\begin{eqnarray}
  \nonumber
  \int\ub(\Delta-\lambda) u + \int mv\ub &=& \int f \ub
\\\na{and integrate by parts}\nonumber
 -\int|\del u|^{2} -\lambda\int|u|^{2}  &=&\int f \ub -
 \int mv\ub 
\\\label{eq:22}
 \int|\del u|^{2} + (\lambda-K)\int|u|^{2}
&\le& K\left(||f||^{2} + ||v||^{2}\right)
\end{eqnarray}
so that
\begin{eqnarray}\label{eq:23}
  ||u||^{2}&\le&\frac{K}{\lambda-K}\left(||f||^{2}
    + ||v||^{2}\right)
\\\na{and}\label{eq:34}
||\del u||^{2}&\le&K\left(||f||^{2}
    + ||v||^{2}\right)
\\\na{Using \eqref{eq:23} in \eqref{eq:7}, gives}
\label{eq:8}
||v||^{2}&\le& K\left(||f||^{2} +
  \frac{||g||^{2}}{\lambda}\right)
\\\na{so that  (\ref{eq:23}) and (\ref{eq:34}) become}
  \label{eq:25}
  ||u||^{2}&\le&\frac{K}{\lambda-K}\left(||f||^{2}
    +\frac{||g||^{2}}{\lambda}\right)
\\\na{and}\label{eq:6}
||\del u||^{2}&\le&K\left(||f||^{2}
    +\frac{||g||^{2}}{\lambda}\right)
\\\na{It now follows from \eqref{eq:2} that}
\label{eq:9}
||\Delta u||&\le& K\left(||f|| +
  \frac{||g||}{\lambda}\right)
\\\na{ and from \eqref{eq:4} that}
\label{eq:20}
||\Delta v||&\le& K\left(\lambda||f|| + ||g||\right)
\\\na{It also folows from  \eqref{eq:18} that\espace}
\label{eq:21}
\int|\del(\phi v)|^{2}&\le& K\left(||f||^{2}+
  ||g||^{2}\right)
\end{eqnarray}

The proposition below is a direct consequence of this list
of inequalities.
\begin{proposition}\label{sec:prop6} For \(\lambda\) real, positive, and
large enough,
  \begin{enumerate}
  \item 
\(\displaystyle
  \left(B-\lambda I\right):H^{2}_{0}(D)\oplus
\left\{v\in\ltd : \Delta v\in\ltd\right\}
\longrightarrow \ltd\oplus\ltd
\)\\
 is invertible.

\item If we write the the resolvent in block diagonal form,
    \begin{equation}
      \label{eq:11}
      \left(B-\lambda I\right)^{-1} =  \begin{pmatrix}
                  R_{11}&R_{12}\\
                     R_{21}&R_{22}
                  \end{pmatrix} 
    \end{equation}
    then \(R_{11}\), \(R_{12}\), and  \(R_{22}\) are compact.
    If
    \(\phi(x)\) is a smooth function vanishing in a
    neighborhood of \(\bd{D}\) and equal to 1 on
    \(D\setminus N\), them \(\phi R_{21}\) is compact.
  \item 
    \(\displaystyle
      ||R_{11}||+ ||R_{12}||+ ||R_{22}||+||\phi R_{21}||
              \le\frac{K}{\lambda}
    \)
  \end{enumerate}
\end{proposition}
\begin{proof}

The combination of \eqref{eq:8} through \eqref{eq:20}
show that
\(B-\lambda I\) is one to one and has closed range. If we note that
the adjoint of \(B\) is 
\begin{eqnarray}
  \label{eq:10}
  \left(B-\lambda I\right)^{*} = \begin{pmatrix}
                  \D-\lambda I&\bar{m}\\
                     0&\Doo-\lambda I
                  \end{pmatrix} =
                  \begin{pmatrix}
                    0&1\\
                    1&0
                  \end{pmatrix} \bar{\left(B-\lambda I\right)}\begin{pmatrix}
                    0&1\\
                    1&0
                  \end{pmatrix}
\end{eqnarray}
we see that \(B^{*}\) is just \(B\) with \(m\) replaced
by \(\bar{m}\) and \(u\) and \(v\) interchanged. Therefore
\eqref{eq:8} and \eqref{eq:25} ensure uniqueness for
\(B^{*}\) and consequently existence for \(B\). Now
\(R_{11}f\) in \eqref{eq:11} denotes the function \(u\)
that satisfies \eqref{eq:2} and \eqref{eq:4} with
\(g=0\), and \(R_{12}g\) denotes the function \(u\) that
satisfies \eqref{eq:2} and \eqref{eq:4} with \(f=0\). The
estimates \eqref{eq:25} and \eqref{eq:6} imply that both are
compact.\\

Similarly, \(R_{22}g\) denotes the function \(v\) that
satisfies \eqref{eq:2} and \eqref{eq:4} with \(f=0\) and
\(R_{21}f\)  denotes the function \(v\) that satisfies
\eqref{eq:2} and \eqref{eq:4} with \(g=0\) The combination
of \eqref{eq:8} and \eqref{eq:21} imply the compactness
of \(\phi R_{21}\) and \(\phi R_{22}\).  To see the
compactness of \(R_{22}\),
suppose that we have a sequence \(\{g_{n}\}\) converging
weakly to zero. Now \(u_{n}=R_{12}g_{n}\) and \(\phi v_{n}
=\phi R_{22}g_{n}\)  converge strongly to zero.
According to \eqref{eq:12},
\begin{eqnarray}
  \label{eq:13}
  \int m|v_{n}|^{2} = \int \bar{g_{n}}u_{n}
\end{eqnarray}
and the right side converges to zero.  Hence
\begin{eqnarray*}
  \int m (1-\phi^{2})|v_{n}|^{2} = \int \bar{g_{n}}u_{n} -
  \int \phi^{2} m |v_{n}|^{2} \longrightarrow 0
\end{eqnarray*}
and therefore  the coercivity condition on \(m\) implies that 
 \(||(1-\phi)v_{n}||\) converge to zero
 so that \(R_{22}\) is compact.\\


 The estimates in the last item follow from
 \eqref{eq:25},\eqref{eq:8}, and the combination of
 \eqref{eq:8} and \eqref{eq:17}.
\end{proof}

 We have constructed the resolvent \((B-\lambda I)^{-1}\)
 for large positive \(\lambda\) and have shown that it is
 \textit{upper triangular compact}.
 
\section{Upper Triangular Compact Operators}
\label{sec:utc}

For an unbounded operator with a compact resolvent,
discreteness of spectra follows from the resolvent
identity

\begin{eqnarray}
  R(\mu) - R(\lambda) = (\mu - \lambda)R(\mu)R(\lambda)
\\\noalign{rewritten as}\label{eq:27}
  R(\mu) = R(\lambda)\left(I - (\mu - \lambda)R(\lambda)\right)^{-1}
\end{eqnarray}

and the analytic Fredholm theorem \cite{Colton-Kress2}. 

 \begin{proposition}[The Analytic Fredholm Theorem]\label{th:aft}
   Suppose that \(R(\lambda)\) is an analytic compact
   operator valued function of \(\lambda\) for \(\lambda\)
   in some open connected set \(L\). Then if
   \(I-R(\lambda_{0})\) is invertible for one \(\lambda_{0}\in
   L\), it is invertible for all but a discrete set of
   \(\lambda\in L\).
 \end{proposition}

 \begin{remark}
   Because \(I-R(\lambda)\) must have index \(0\), it is
   enough to check that \(\ker(I-R(\lambda_{0}))\) or
   \(\coker(I-R(\lambda_{0}))\) is empty. We don't need to
   check both.
 \end{remark}

\begin{definition}
  Suppose that \(R\) is a bounded operator mapping a
  Hilbert space \(H\) to itself.  If the Hilbert space has
  a decomposition into a direct sum
  \(H=\Oplus_{j=1}^{n}H_{j}\), we say that \(R\) is
  \textbf{upper triangular compact} (UTC) (with respect to
  this decomposition) if the upper
  triangular blocks (including the diagonal) in the
  corresponding decomposition of \(R =
  \Sum_{jk=1}^{n}R_{jk}\) are compact.
 \end{definition}

 The proposition below asserts that the
 conclusions of the analytic Fredholm theorem continue to
 hold for operators that are upper triangular Fredholm.

\begin{proposition}[The Upper Triangular Analytic Fredholm Theorem]
  Suppose that \(R(\lambda)\) is an analytic UTC operator
  valued function of \(\lambda\) for \(\lambda\) in an
  open connected set \(L\).  Then if \(I-R(\lambda_{0})\)
  has empty kernel or cokernel for one \(\lambda_{0}\), it
  is invertible for that
  \(\lambda_{0}\) and for all but a discrete set of \(\lambda\in L\).\\
 \end{proposition}
 \begin{proof}

   We perform block Gaussian elimination modulo compact
   operators. We subtract a multiple of the first row from
   each subsequent row, so that the resulting operators in
   the first column are compact.
   Specifically, we set
   \begin{eqnarray*}
     \widetilde{R_{n1}} &=&
     R_{n1}-R_{n1}\left(I-R_{11}\right)
     \\
     &=& R_{n1}R_{11}
    \\\na{and, for \(m>1\),}
     \widetilde{R_{nm}} &=& R_{nm}-R_{n1}R_{1m}
   \end{eqnarray*}

   The new first column is compact below the diagonal, and
   the entries in the other columns have only been changed
   by the addition of a compact operator ( \(R_{1m}\) for
   \(m>1\) are compact), so the new operator still has the
   form, identity plus UTC. After repeating for each
   subsequent column, we reach a point where all off
   diagonal blocks are compact. We have produced a
   factorization 
   \begin{eqnarray*}
     I - R(z) = (I - L(z))(I - C(z))
   \end{eqnarray*}
   with \(L(z)\) strictly lower triangular, and \(C(z)\)
   compact. The first factor is always invertible and we
   may invoke the analytic Fredholm theorem \ref{th:aft},
   and the remark which follows it, on the second factor.
\end{proof}

\begin{theorem}[UTC Resolvent Theorem]\label{sec:utc-resolvent}
  Let \(B\) be a closed densely defined operator on an
  infinite dimensional Hilbert space and suppose that for
  one complex number \(\lambda_{0}\), \((B-\lambda_{0}
  I)\) is invertible and \((B-\lambda_{0} I)^{-1}\) is
  UTC. Then the spectrum of \(B\) consists of a (possibly empty)
  discrete set of eigenvalues. with finite dimensional  generalized
  eigenspaces.
\end{theorem}

\begin{proof}
  According to the resolvent identity \eqref{eq:27},
for any  complex number \(\tau\), we may write
   \begin{eqnarray*}
     \left(B-\tau I\right) =
     \left(B-\lambda_{0} I\right)\left(I - \left(B-\lambda_{0} I\right)^{-1}(\tau-\lambda_{0})\right)
   \end{eqnarray*}
Because \(\left(B-\lambda_{0} I\right)^{-1}\) is UTC, the
UT analytic Fredholm theorem implies that
factor on the right is invertible at all 
but a discrete set of  points \(\tau_{n}\), and the dimension of
the kernel is finite at all such points. 
\end{proof}

\section{Application of  the UT Analytic Fredholm Theorem}

 \begin{proof}[Proof of Theorem \ref{sec:th2}]
   Let \(B\) be the operator defined in
   \eqref{eq:1}. Proposition \ref{sec:prop6} tells us that,
   if we choose \(\lambda_{0}\) real, positive, and large
   enough, then \(\left(B-\lambda_{0} I\right)\) is
   invertible and \(\left(B-\lambda_{0} I\right)^{-1}\) is
   UTC, so
   theorem~\ref{sec:utc-resolvent} guarantees that the
   spectrum of \(B\) is discrete and of finite multiplicity.
\end{proof}

\begin{proof}[Proof of Theorem \ref{sec:th1}]

\begin{eqnarray}\label{eq:3}
    (B-\lambda_{0} I_{m}) &=& \left(B-\lambda_{0} I\right)
    \left(I - \lambda_{0}\left(B-\lambda_{0} I\right)^{-1}
      \begin{pmatrix}
        m&0\\0&0
      \end{pmatrix}\right)
  \end{eqnarray}

The factor on the right is of the form identity plus UTC,
and therefore UT~Fredholm of index zero. We will  show
that the kernel or cokernel of that factor is empty for a
large positive  \(\lambda_{0}\).  It will then follow that
\begin{eqnarray*}
  (B-\lambda_{0} I_{m})^{-1} =  \left(I - \lambda_{0}\left(B-\lambda_{0} I\right)^{-1}
      \begin{pmatrix}
        m&0\\0&0
      \end{pmatrix}\right)^{-1}
\left(B-\lambda_{0} I\right)^{-1}
\end{eqnarray*}
is UTC -- because the product of UTC and identity minus
UTC is UTC. Multiplication by \(I_{m}\) also preserves
UTC, so \(\left(I_{m}^{-1}B -\lambda_{0} I\right)^{-1}\) is
the upper triangular compact resolvent of \(I_{m}^{-1}B\) at
\(\lambda_{0}\). Theorem ~\ref{sec:utc-resolvent} now
implies theorem \ref{sec:th1}.
Therefore we may finish the proof of theorem~\ref{sec:th1}
with:

\begin{proposition}\label{sec:appl-ut-analyt}
  Suppose that \(\Re(1+m(x))\ge \delta>0\) in 
  \(D\). 
  \begin{enumerate}
  \item If \(m\) is real in \(D\) and \(m <-m_{*}<0\) in
    some neighborhood \(N\) of \(\bd{D}\), then, for
    \(\lambda\) real and sufficiently large,
    \(\ker(B-\lambda I_{m})\) is empty.
  \item If \(\frac{-\pi}{2}<\theta<\frac{\pi}{2}\) and
    \(\Re(e^{i\theta}m(x))>m_{*}>0\) in some neighborhood
    \(N\) of \(\bd{D}\), then, for \(\lambda\) real and
    sufficiently large, \(\coker(B-\lambda I_{m})\) is
    empty.
  \end{enumerate}
\end{proposition}

\begin{proof}
The kernel of \((B-\lambda I_{m})\) consists of
functions satisfying
\begin{eqnarray}
    \label{eq:14}
    (\Delta-\lambda) u + mv &=& \lambda m u
\\\label{eq:15}
    (\Delta-\lambda) v      &=& 0
  \end{eqnarray}
  with \(u\) and \(\ddnu{u}\) vanishing on \(\bd{D}\).
We multiply the conjugate of \eqref{eq:15} by \(u\) and 
integrate by parts to obtain
  \begin{eqnarray}\nonumber
    \int \vb(\Delta-\lambda)u &=& 0
\\\na{Subtracting this  from the integral of \(\vb\) times
  the \eqref{eq:14} yields\espace}\nonumber
\int m|v|^{2} &=& \lambda\int mu\vb
\\\na{Because the left hand side is real\espace}\label{eq:16}
\int m|v|^{2} &=& \lambda\int m\ub v
  \end{eqnarray}
We next multiply \eqref{eq:14} by \(\ub\) and integrate by
parts to find that
\begin{eqnarray}\nonumber
  -\int|\del u|^{2} -\lambda\int(1+m)|u|^{2} &=& -\int
  m\ub v
\\\na{which becomes, after inserting \eqref{eq:16}\espace}
\nonumber
 -\int|\del u|^{2} -\lambda\int(1+m)|u|^{2} 
&=& \frac{-1}{\lambda}\int m|v|^{2}
\\\label{eq:26}
&=& \frac{-1}{\lambda}\left(\int(1- \phi^{2})m|v|^{2} +
  \int\phi^{2} m|v|^{2}\right)
\end{eqnarray}
The estimate \eqref{eq:19} with \(g=0\) and \(\rho=1\),
implies that the real number
\begin{eqnarray}\label{eq:33}
  z &=& \frac{\int\phi^{2} m|v|^{2}}{\int(1- \phi^{2})m|v|^{2}}
\\\na{satisfies}\label{eq:35}
|z| &\le& \frac{K}{\lambda}
\\\na{so that for \(\lambda\) large enough}\nonumber
  1+z&>&0
\end{eqnarray}
If we rewrite \eqref{eq:26} as
\begin{eqnarray*}
  -\int|\del u|^{2} -\lambda\int(1+m)|u|^{2} &=&
  \frac{-1}{\lambda}\left(\int(1- \phi^{2})m|v|^{2}\right)
\left(1+z\right)
\end{eqnarray*}

we see that, for \(\lambda\) large enough, the
 right hand side is
positive, while the 
left hand side is negative, unless both \(u\) and \(v\)
are identically zero.\\

Finally, the kernel of \((B-\lambda I_{m})^{*}\)
consists of functions satisfying
\begin{eqnarray}
   \label{eq:28}
    (\Delta-\lambda) v + \mb u     &=& 0
 \\\label{eq:29}
    (\Delta-\lambda) u  &=& \lambda\mb u
  \end{eqnarray}
  with \(v\) and \(\ddnu{v}\) vanishing on
  \(\bd{D}\). 
We multiply \eqref{eq:28} first by \(\vb\) and integrate
  \begin{eqnarray}
-\int|\del v|^{2} -\lambda\int |v|^{2} 
&=& -\int \mb u\vb
\\\na{and then multiply the conjugate of \eqref{eq:28} by
  \(u\) to obtain}\nonumber
\int u(\Delta-\lambda)\vb &=& -\int m |u|^{2}
\\\na{Multiplying \eqref{eq:29} by \(\vb\) and
  integrating by parts gives}\nonumber
\int u(\Delta-\lambda)\vb &=& \lambda\int \mb u\vb
\\\na{Combining gives}\nonumber
  -\int|\del v|^{2} -\lambda\int |v|^{2} 
&=& \frac{1}{\lambda}\int m|u|^{2}
\\\na{we split the integral on the right into two parts}\nonumber
&=& \frac{1}{\lambda}\left(\int(1-\phi^{2})m|u|^{2}
+ \int\phi^{2} m|u|^{2}\right)
\\\na{We again employ \eqref{eq:19} with \(g=0\) and
  \(\rho=1+m \), to see that\espace}\label{eq:30}
-\int|\del v|^{2} -\lambda\int |v|^{2} 
&=&\frac{1}{\lambda}\left(\int(1-\phi^{2})m|u|^{2}\right)(1+z)
  \end{eqnarray}

  where \(z\) defined as in \eqref{eq:33}, is complex, but
  still satisfies (\ref{eq:35}). This time, the
  left hand side is a negative real number, and the
  hypothesis guarantees that, for every \(x\), \(m(x)\)
  sits in an open half plane (a cone) that does not
  contain the negative real semi-axis. This means that the
  argument of the integral on the right hand side of
  \eqref{eq:30} is bounded away from \(\pi\). The argument
  of \(1+z\) approaches zero as \(\lambda\) increases, so
  for large enough \(\lambda\), the left hand side belongs
  to the negative real axis, while the right hand side
  cannot, unless \(u\), and therefore also \(v\), is
  identically zero.
\end{proof}
This finishes the proof of theorem \ref{sec:th1}.
\end{proof}

\begin{proof}[Proof of Proposition \ref{sec:spectralP}]

  Those familiar with sprectral theory will recognize that
  the main step in the proof we give below verifies that,
  if \(p\) and \(m\) are two different contrasts, then
  \(B_{p}\) (resp. \(I_{p}^{-1}B_{p}\)) is arelatively
  bounded perturbation of \(B_{m}\)
  (resp. \(I_{m}^{-1}B_{m}\)).\\

  \newcommand{\rbm} {\ensuremath{\left(B_{m}-\lambda
        I\right)^{-1}}} \newcommand{\rbn}
  {\ensuremath{\left(B_{p}-\lambda I\right)^{-1}}}
  \newcommand{\rtm}
  {\ensuremath{\left(I_{m}^{-1}B_{m}-\lambda
        I\right)^{-1}}} \newcommand{\rtn}
  {\ensuremath{\left(I_{p}^{-1}B_{p}-\lambda
        I\right)^{-1}}} 

Because the curve \(\gamma\) avoids the spectrum of
\(B_{m}\), the resolvent \rbm is a
  holomorphic, hence continuous, function on the compact
  set \(\gamma\). Therefore
  \begin{eqnarray*}
    \Gamma(m) := \sup_{\lambda\in\gamma}|||\rbm||| <\infty
  \end{eqnarray*}
  where we have used \(|||A|||\) denote the norm of the operator
  as a mapping from \(\ltd\oplus\ltd\) into
  \(H^{2}_{0}(D)\oplus\ltd\) and will  use \(|A|\) to
  denote the norm of  a mapping from
  \(H^{2}_{0}(D)\oplus\ltd\) to \(\ltd\oplus\ltd\). Now,

  \begin{eqnarray*}
    |B_{p} - B_{m}| &=& \left| \begin{pmatrix}
                      0&(p-m)\\
                      0&0
                        \end{pmatrix}\right|
\le ||p-m||_{\infty}
\\\na{so that \espace}
|||\rbn||| &=& |||\rbm \left(I-(B_{p} -
  B_{m})\rbm\right)^{-1}|||
\\
&\le& \Gamma(m)\left(1 - ||p-m||_{\infty}\Gamma(m)\right)^{-1}
  \end{eqnarray*}
so we may conclude the existence of \rbn for
\(||p-m||_{\infty}< 1/\Gamma(m)\) and further that 

\begin{eqnarray*}
  |||\rbn-\rbm||| &=& |||\rbn\left(B_{p} -
    B_{m}\right)\rbm|||
  \\
  &\le& |||\rbn|||\;|B_{p} - B_{m}|\;|||\rbm|||
  \\
  &\le&\frac{\Gamma(m)}{(1 - ||p-m||_{\infty}\Gamma(m))}
  \;||p-m||_{\infty}\;\Gamma(m)
  \\\na{and hence, letting \(|\gamma|\) denote the length
    of \(\gamma\) and recalling the definition of the spectral projection from
    (\ref{eq:36})\espace} 
  |||P_{\gamma}(p) - P_{\gamma}(m)|||&\le&
\frac{|\gamma|\;||p-m||_{\infty}\;\Gamma^{2}(m)}{(1 - ||p-m||_{\infty}\Gamma(m))}
\end{eqnarray*}

which establishes the continuity of \(P_{\gamma}\) in the
Born approximation.\\

The proof for  \(I_{m}^{-1}B_{m}\) is analogous. We redefine

\begin{eqnarray*}
    \Gamma(m) := \sup_{\lambda\in\gamma}|||\rtm||| <\infty
  \end{eqnarray*}
and compute the norm of the  perturbation

\begin{eqnarray*}
    |I_{p}^{-1}B_{p} - I_{m}^{-1}B_{m}| &=& \left| \begin{pmatrix}
                     \frac{m-p}{(1+m)(1+p)}\Delta_{00} &\frac{p-m}{(1+m)(1+p)}\\
                      0&0
                        \end{pmatrix}\right|
\le \frac{||p-m||_{\infty}}{\delta^{2}}
\\\na{so that \espace}
|||\rtn||| &=& |||\rtm \left(I-(I_{p}^{-1}B_{p} -
  I_{m}^{-1}B_{m})\rtm\right)^{-1}|||
\\
&\le& \Gamma(m) \left(1 - \frac{||p-m||_{\infty}}{\delta^{2}}\Gamma(m)\right)^{-1}
  \end{eqnarray*}
and
\begin{eqnarray*}
  |||\rtn&-&\rtm|||\\
 &=& |||\rtn\left(I_{p}^{-1}B_{p} -
    I_{m}^{-1}B_{m}\right)\rtm|||
  \\
  &\le& |||\rtn|||\;|I_{p}^{-1}B_{p} - I_{m}^{-1}B_{m}|\;|||\rtm|||
  \\
  &\le&\frac{\Gamma(m)} {1 - \frac{||p-m||_{\infty}}{\delta^{2}}\Gamma(m)}
  \;\frac{||p-m||_{\infty}}{\delta^{2}}\;\Gamma(m)
\\
  &=&\frac{\Gamma^{2}(m)||p-m||_{\infty} }{\delta^{2} - ||p-m||_{\infty}}
\end{eqnarray*}
which establishes the continuity of the resolvent and
hence the spectral projection.

\end{proof}
\section{Discussion}
\label{sec:discussion}

We have shown that the interior transmission eigenvalue
problem, with some coercivity conditions on the contrast
\(m\), naturally leads to a simple class of closed operators
with upper triangular compact resolvents, which share the
properties of operators with compact
resolvent.\\

For the Born approximation to the interior transmission
eigenvalue problem (theorem~\ref{sec:th2}), the coercivity
condition on the values of \(m\) near the boundary seems
pretty natural. We expect that this cannot be weakened too
much. The conditions we require for theorem~\ref{sec:th1}
are more ad hoc. They are required for our proof, but we
see no strong reason to believe they are necessary.\\

If we could prove that these resolvents were not
quasi-nilpotent, existence of transmission eigenvalues,
and very likely completeness of generalized eigenspaces,
would follow.\\

\bibliographystyle{abbrv}
\bibliography{agest,sylvester,cit}

\end{document}